\documentclass[a4paper]{amsart}

\usepackage{epsfig,amsmath}
\usepackage{tikz}
\usepackage{xspace}
\usepackage[psamsfonts]{amssymb}
\usepackage[utf8]{inputenc}
\usepackage{float}\restylefloat{figure}
\usepackage{color}
\usepackage{hyperref}

\newtheorem*{lemma*}{\bf Lemma}
\newtheorem*{sublemma*}{\bf Sublemma}
\newtheorem*{claim*}{\bf Claim}
\newtheorem*{complement*}{\bf Complement}

\renewcommand{\epsilon}{\varepsilon}

\newcommand{\cal}[1]{\mathcal{#1}}
\newcommand{\ov}[1]{\overline{#1}}

\newcounter{rememberItem}

\def\bEA{\begin{eqnarray*}}
\def\eEA{\end{eqnarray*}}
\def\bEAn{\begin{eqnarray}}
\def\eEAn{\end{eqnarray}}

\def\wt{\widetilde}
\def\wh{\widehat}
\def\on{\operatorname}

\def\cal{\mathcal}

\def\C{{\mathbb C}}

\def\R{{\mathbb R}}

\def\Z{{\mathbb Z}}

\usepackage{gensymb}
\usepackage{pdflscape}

\graphicspath{{img/}}

\newtheorem{lemma}{Lemma}
\newtheorem{proposition}[lemma]{Proposition}

\newtheorem{corollary}[lemma]{Corollary}

\newcommand{\Emb}{\mathrm{Emb}}
\newcommand{\height}{\on{ht}}

\title{Yet another sphere eversion}

\begin{document}

\begin{abstract}
We present a (possibly) new sphere eversion based on the contractibility* of a certain subset of the space of immersions of the circle in the plane. (*: by strong deformation retraction)
\end{abstract}

\maketitle


A sphere eversion is a continuous path in the space of immersions of the sphere in $\R^3$, from an embedded sphere to an embedded sphere with orientation reversed, i.e.\ with the inner and outer sides reversed.
Existence of sphere eversions has been discovered by Smale as a consequence of a stronger theorem. Here is not the place for a history of the subject: the article \cite{S} does it very well.

We decided to split our presentation into two articles. The first and present one focuses on a (possibly) new proof of \emph{existence} of the eversion. It goes quickly to the objective, trying to stay elementary, and without trying at getting the most elegant statements. For this purpose we first present the procedure in an intuitive way in Sections 1 to 4, stating only the most meaningful lemmas/proposition. We state a couple of technical lemmas used in the actual proof in Section 5, and leave all the proofs to the appendix.
The second article will refine a little bit the procedure, answer natural questions that arise, and will try to realize the scheme proposed above with a minimal number of generic transitions.

All the lemmas proved here are already known and their proofs are included only for self-containedness. Their statements and proofs assume that the reader has knowledge of the following bases of differential geometry: differential, manifold, immersion, embedding, implicit function theorem, tubular neighborhood, of topology: homotopy, isotopy, ambient isotopy, and of analysis: partial derivatives, $C^k$ functions, Banach spaces, convolution.

While writing the present article, the author learned that Derek Hacon found in the 1970's an eversion bearing several similarities with the one presented here. This work is unpublished.

\vskip-1cm

\part{The principle}

The existence of an eversion in the set of $C^1$ immersions implies the existence in the $C^k$ class for any $k$, and even\footnote{For instance one can work with $S^2=$ the Euclidean sphere and convolve all immersions using a $C^\omega$ kernel on $S^2\times S^2$.
For a kernel sufficiently condensed near the diagonal, the convolved map remains an immersion.} in $C^\omega$.
Here we choose to work with $C^1$ for simplicity.

We will mention \emph{isotopy} between $C^1$ self diffeomorphisms of $\R^2$: these are continuous paths in the set of $C^1$ diffeomorphisms from $\R^2$ to $\R^2$, endowed with the topology of uniform convergence on compact subsets of the maps and of their derivatives of order $1$.

We will mention \emph{ambient isotopy} between two immersions or two embeddings $f_0$ and $f_1$ of a manifold $M$ into another $N$. These are pairs of isotopies $\phi_t$ of $N$ and $\psi_t$ of $M$, starting from the identity $\phi_0=\on{Id}_N$ and $\psi_0=\on{Id}_M$, such that $f_t=\phi_t\circ f_0\circ \psi_t^{-1}$. In particular $\phi_1$ carries the image of $f_0$ to the image of $f_1$.

\section{Immersed curves in the plane}\label{sec:imm}

We work with \emph{parameterized} curves: they are considered as functions from a parameterizing set to a space, and two curves with the same image and the same parameterizing set are not necessarily equal, even if orientation is preserved.

The reader may be aware that there is an invariant of homotopy for immersions $f$ of the circle in the plane: it is the winding number $n\in\Z$ of the tangent to the curve at $f(s)$ when the point $s$ moves along the oriented circle.
It is well-known that this is the only invariant (Whitney-Graustein theorem, \cite{Wh}): any two immersions with the same winding number are homotopic in the set of immersions.
The following stronger version is also known. We include a proof in the appendix.

\begin{lemma}\label{lem:a} Choose a base point on the circle. For all $n\in\Z$, let $\cal I_n$ be the set of $C^1$ immersions of the circle in $\R^2$ whose tangent vector has winding number $n$, and points horizontally at the base point, to the right.
Then $\cal I_n$ strongly deformation retracts\footnote{I.e.: For all $\alpha\in\cal I_n$ there exists a homotopy $h_t$ from $h_0=$ the identity on $\cal I_n$ to $h_1=$ the constant map with value $\alpha$, and such that $\forall t\in[0,1]$, $h_t(\alpha)=\alpha$.} to any of its points.
\end{lemma}

The set $\cal I_n$ defined above is an open subset of the Banach space $C^1(S^1,\R^2)$, where the norm is $\max(\sup\|f\|,\sup\|f'\|)$. The lemma refers to this topology. 

\section{Height preserving sphere immersions}\label{sec:tra}

Let $(x,y,z)\in\R^3$.
Let us call horizontals the planes $z=\ $constant. The coordinate $z$ is called the \emph{height} and if $P=(x,y,z)$ we denote $\height(P)=z$.
As a model of the sphere $S^2$ we consider a Euclidean sphere in $\R^3$, seen as a $C^1$ (sub)manifold.
Let us say that an immersion $f:S^2\to\R^3$ is height preserving if $\forall P\in S^2$, $\height(f(P))=\height(P)$.

Then the function $\height\circ f$ has only two critical points on $S^2$: the north and the south poles. Let us call $z_{\max}$ and $z_{\min}$ their heights. Between, each horizontal plane cuts the immersed sphere in a $C^1$ immersed circle. Indeed, the intersection of $f(S^2)$ with the horizontal plane at height $z$  is the image by $f$ of a Euclidean circle: the points of $S^2$ at height $z$. 
Projecting this curve to $\R^2$ by $(x,y,z)\mapsto (x,y)$, we get an immersed $C^1$ curve in the plane that moves continuously with the height (and even better, since the map $f$ is $C^1$ as a map of two variables\footnote{It is however not true that the map $z\mapsto$ the curve is $C^1$ as a function of an interval to the Banach space of $C^1$ functions from $S^1$ to $\R^2$: this would require the existence of a crossed derivative.}).
In neighborhoods of $z_{\min}$ and $z_{\max}$, $f$ must be an embedding and the horizontal slices are $C^1$ embedded circles that shrink down to a point. Let us call \emph{caps} such parts $f(S^2)\cap``z<z_{\min}+\epsilon"$ and $f(S^2)\cap``z>z_{\max}-\epsilon"$.

Lemma~\ref{lem:a} is the main tool in the proof of:

\begin{proposition}\label{prop:d}Let $S^2\subset\R^3$ be a Euclidean sphere. Let $\cal I_H$ denote the set of height preserving $C^1$ immersions of $S^2$ in $\R^3$. Any element of $\cal I_H$ is homotopic in $\cal I_H$ to an embedding. 
\end{proposition}

See the appendix for a proof.
Since we stay in $\cal I_H$, the orientation induced by $S^2$ on the caps is preserved during the homotopy.

\medskip

Conversely if a family of $C^1$ immersions $\gamma_z: S^1 \to \R^2$ with $[z\in[z_0,z_1]$ is given, such that the map $(s,z)\in S^1\times[z_0,z_1] \to \gamma_t(s)$ is $C^1$ then we can look at its \emph{trace}, defined as follows. Let $\cal{G}: (s,z)\in S^1\times[z_0,z_1] \mapsto (\gamma_z(s),z)\in\R^2\times\R \simeq \R^3$. 
This is a height preserving immersion of the cylinder in $\R^3$.
If the top and the bottom curve are embeddings of $S^1$, then we can close the immersion with $C^1$ caps to get a height preserving immersion of a sphere (see Lemma~\ref{lem:recap}).
Let us call this the \emph{capped trace}.

\section{A basic move for curves}

Here we look at particular immersions of the interval and the square instead of the circle and the sphere.
They are subsets of the objects we will encounter below so it is worth focusing on them.
Let us begin with a horizontal segment in the plane. It can be deformed, with both ends being fixed and with the derivative at the ends remaining constant, into a horizontal curve with two loops added, one above and one below, like on Figure~\ref{fig:basicmove}.
Existence of such a move follows from similar methods as in the proof of Lemma~\ref{lem:a} but it is not too hard to find explicit ways of doing it, especially if we want the trace in $\R^3$ to be simple. In the $C^1$ category, we may work with Bezier curves, for instance.

\begin{figure}[htbp]
\scalebox{0.6}{
\begin{tikzpicture}
\node at (0,0) {\includegraphics[bb=0 100 400 300,scale=0.4]{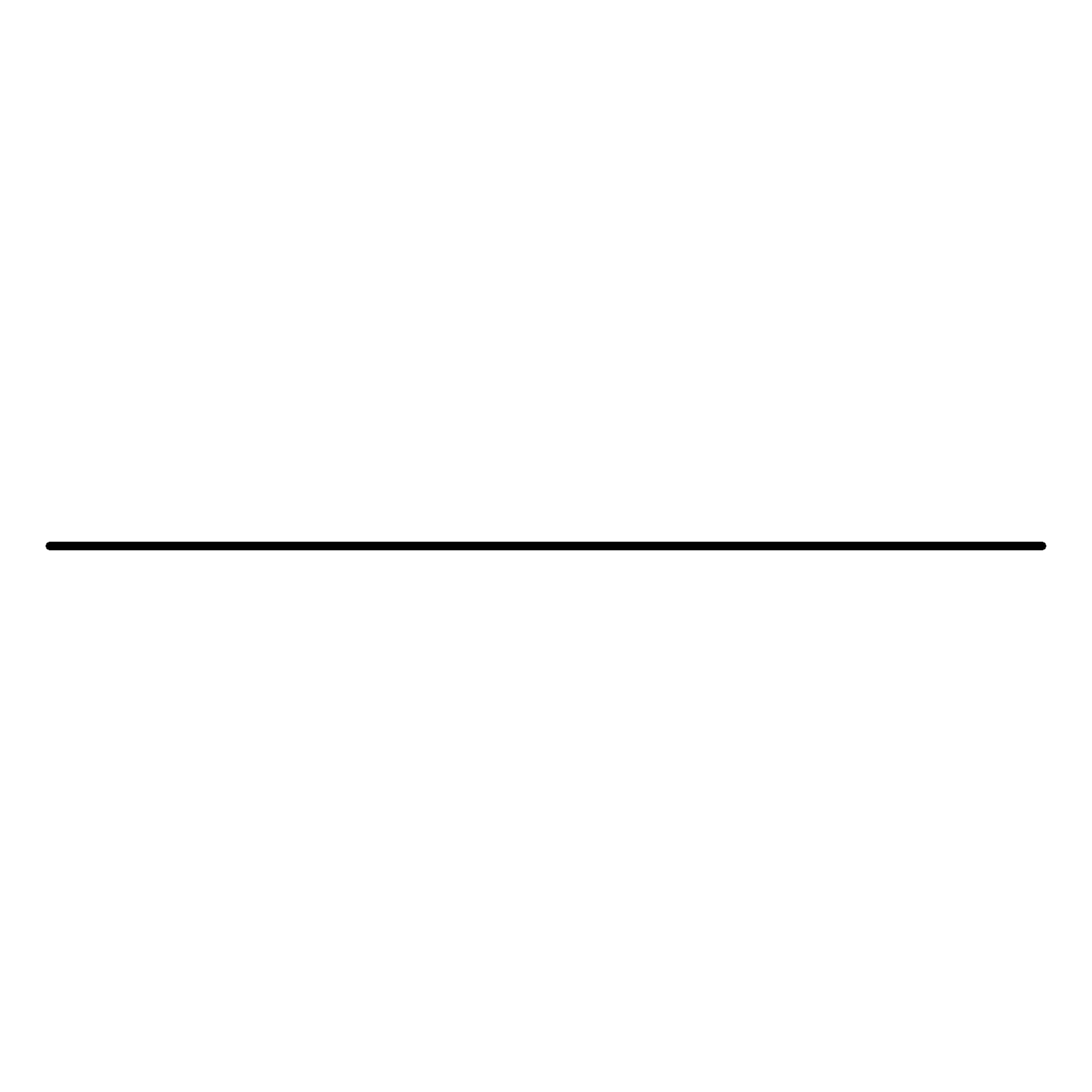}};
\node at (6,0) {\includegraphics[bb=0 100 400 300,scale=0.4]{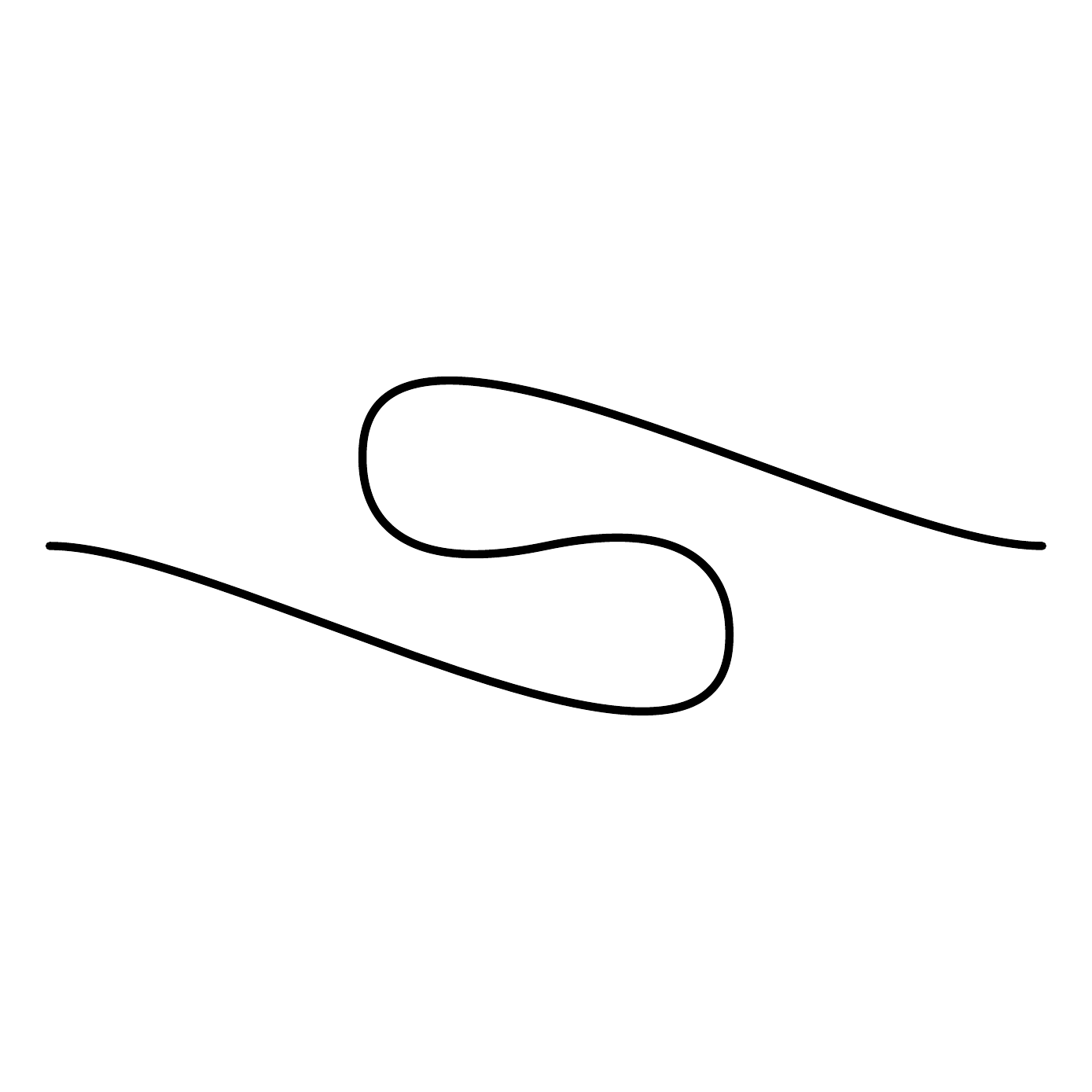}};
\node at (0,-2.2) {\includegraphics[bb=0 100 400 300,scale=0.4]{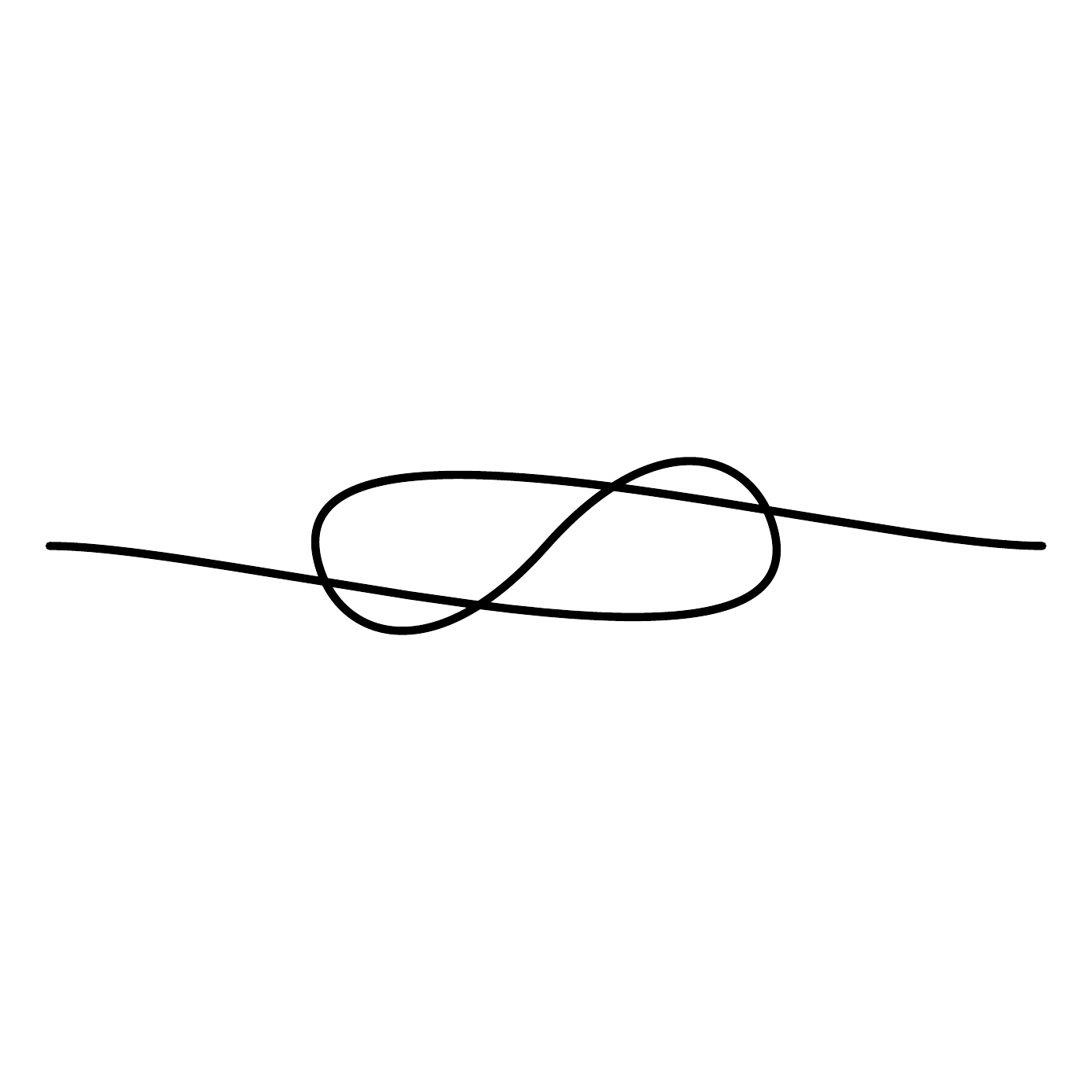}};
\node at (6,-2.2) {\includegraphics[bb=0 100 400 300,scale=0.4]{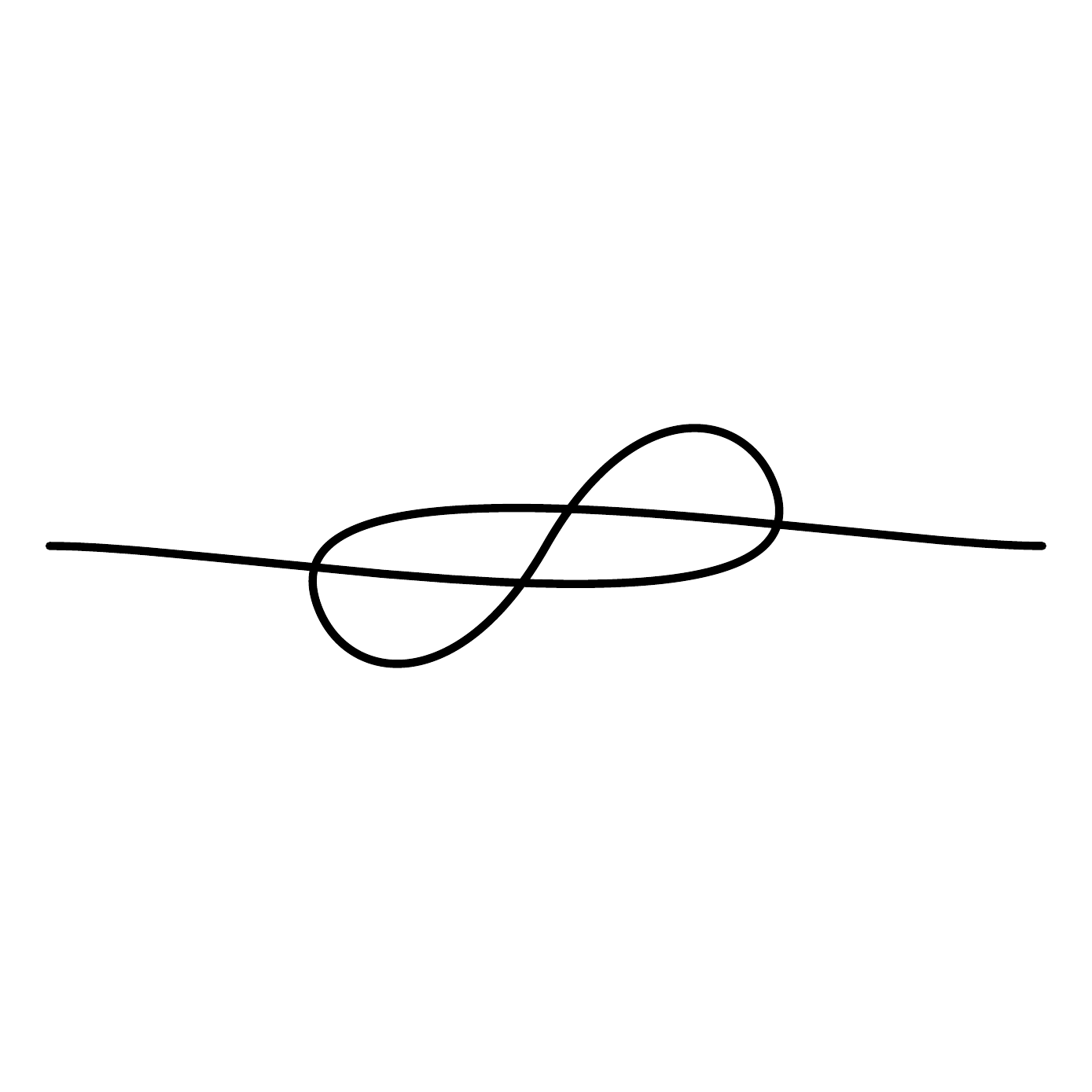}};
\node at (0,-4.5) {\includegraphics[bb=0 100 400 300,scale=0.4]{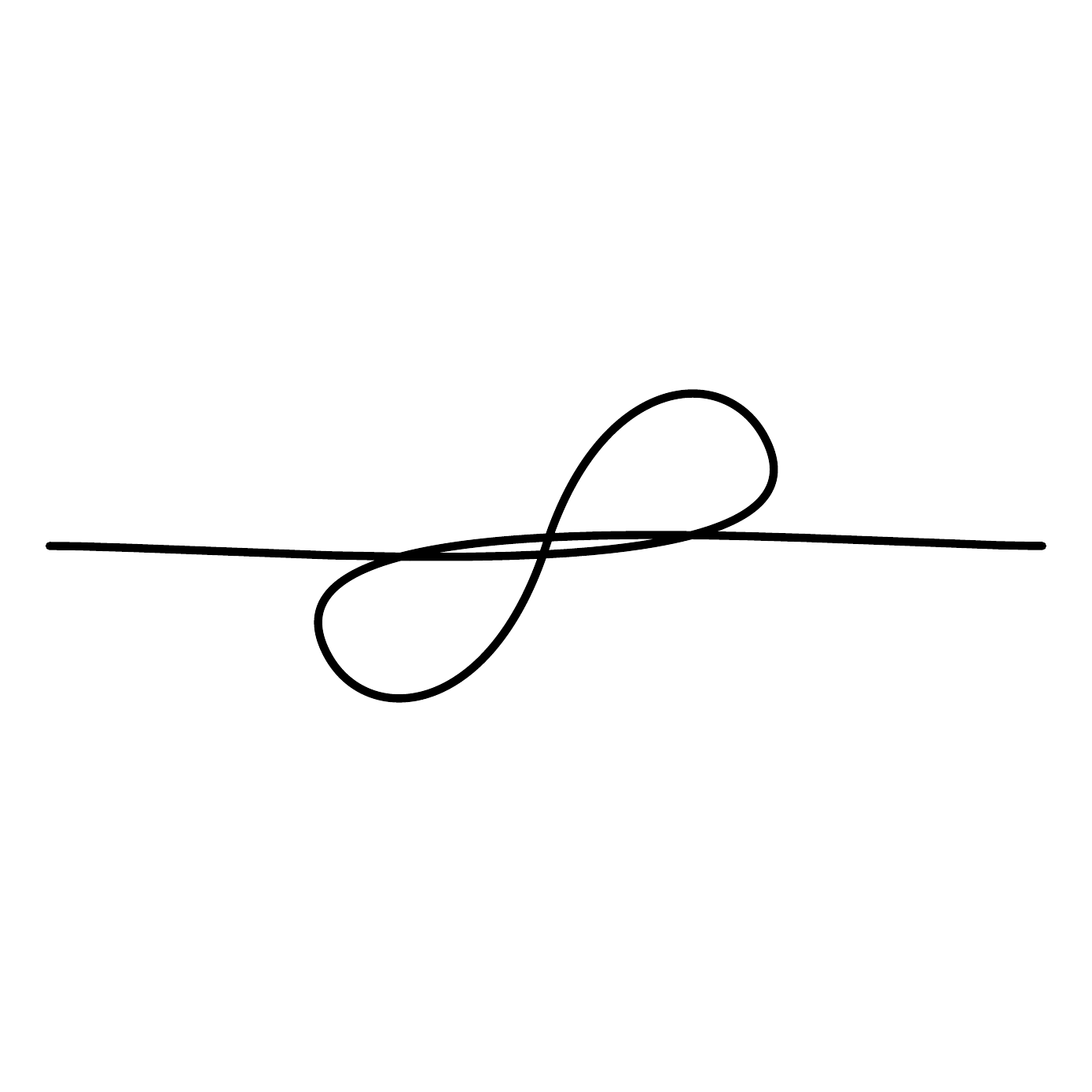}};
\node at (6,-4.5) {\includegraphics[bb=0 100 400 300,scale=0.4]{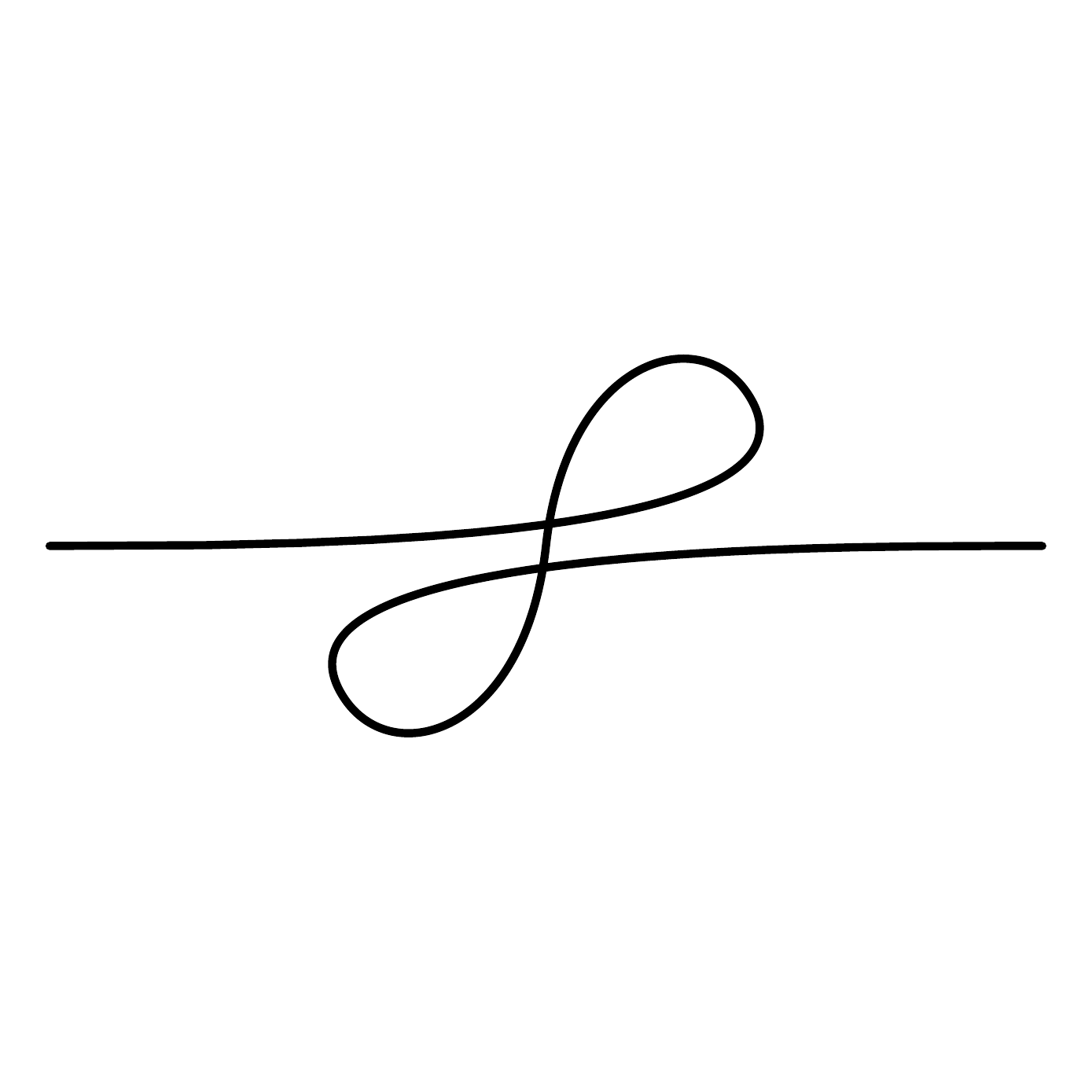}};
\node at (0,-7.7) {\includegraphics[bb=0 100 400 300,scale=0.4]{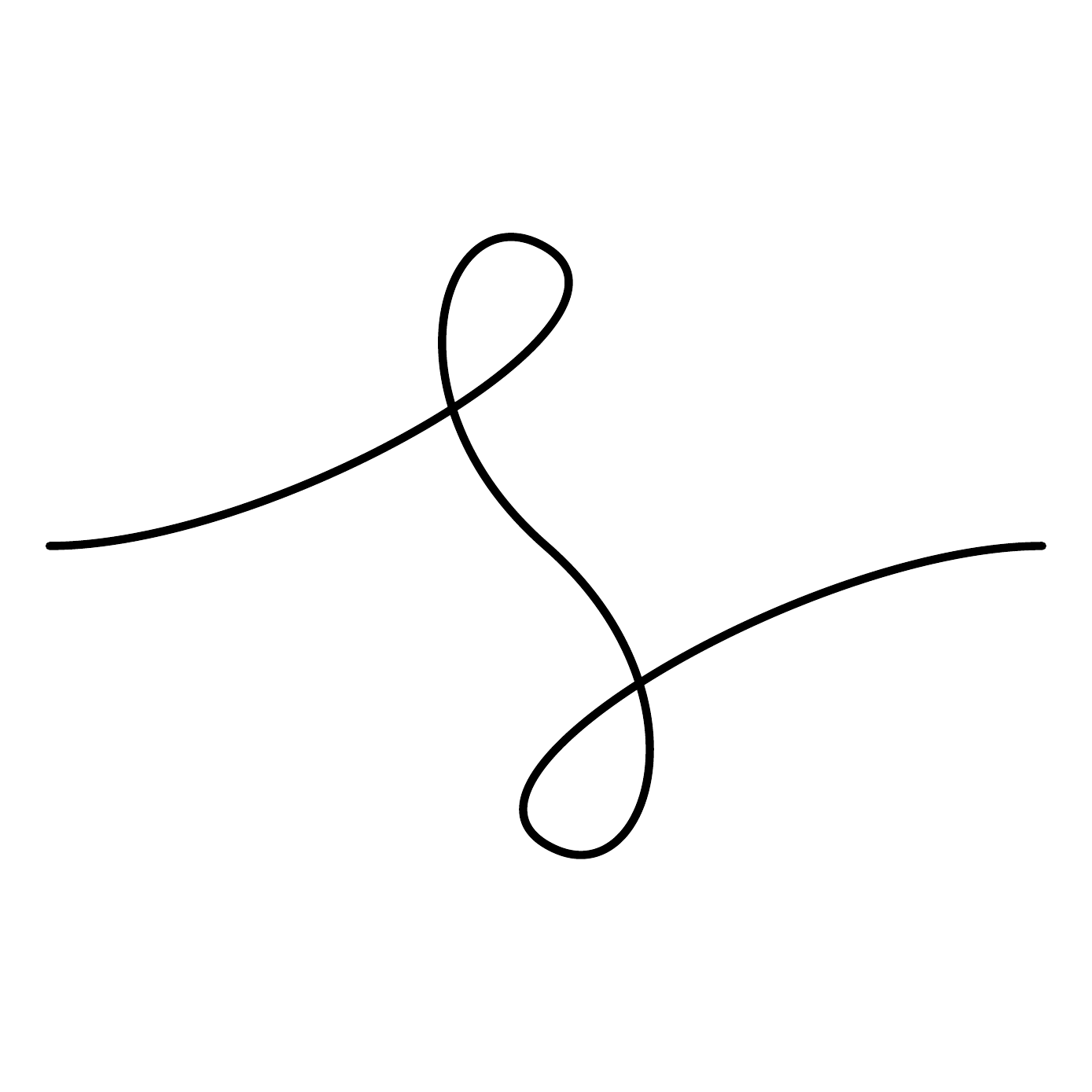}};
\node at (6,-7.7) {\includegraphics[bb=0 100 400 300,scale=0.4]{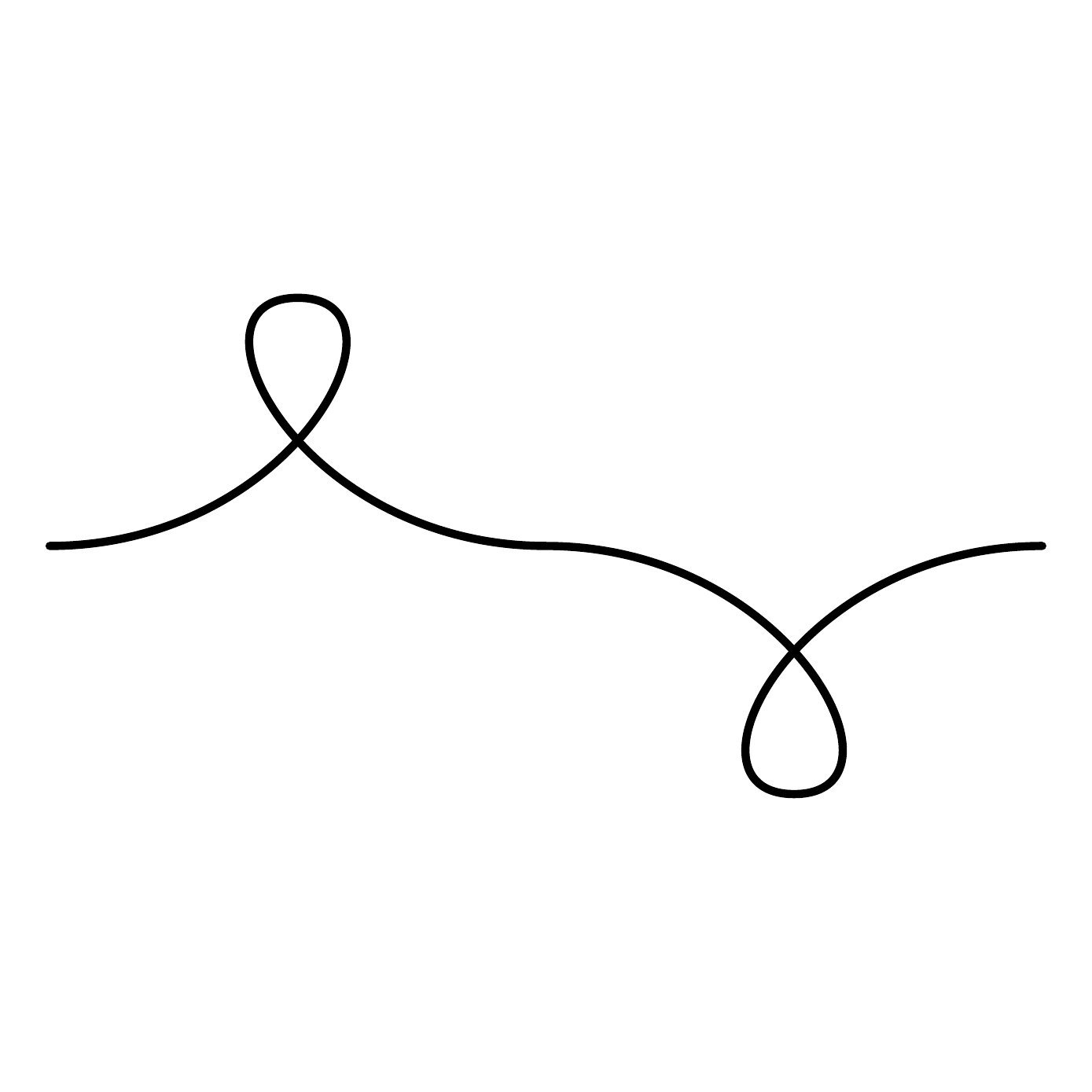}};
\end{tikzpicture}
}
\caption{Creation of a pair of opposite loops on a segment, by deformations fixing the endpoints and their tangents.}
\label{fig:basicmove}
\end{figure}

If the basic move is executed well, the trace has a self-intersection that is a curve shaped like a ribbon tie: there is a triple point from which stem two loops and two longer and open curves. See Figure~\ref{fig:basicsurf}.

\begin{figure}[htbp]
\includegraphics[height=9cm]{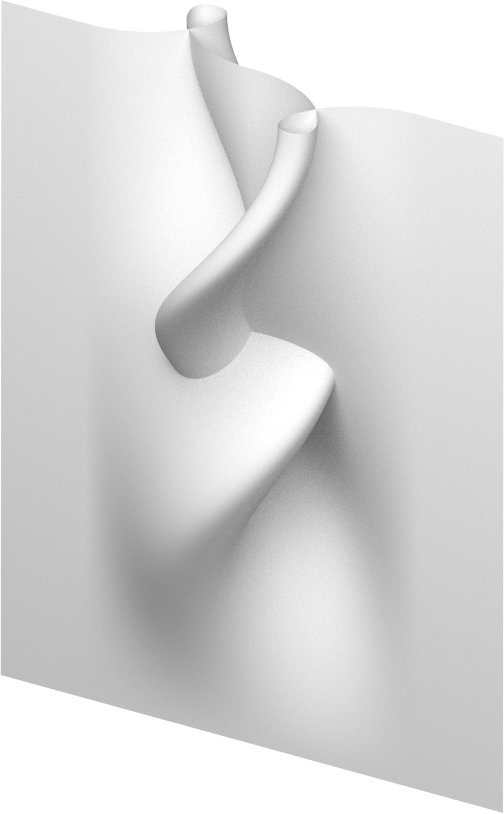}
\quad
\includegraphics[height=9cm]{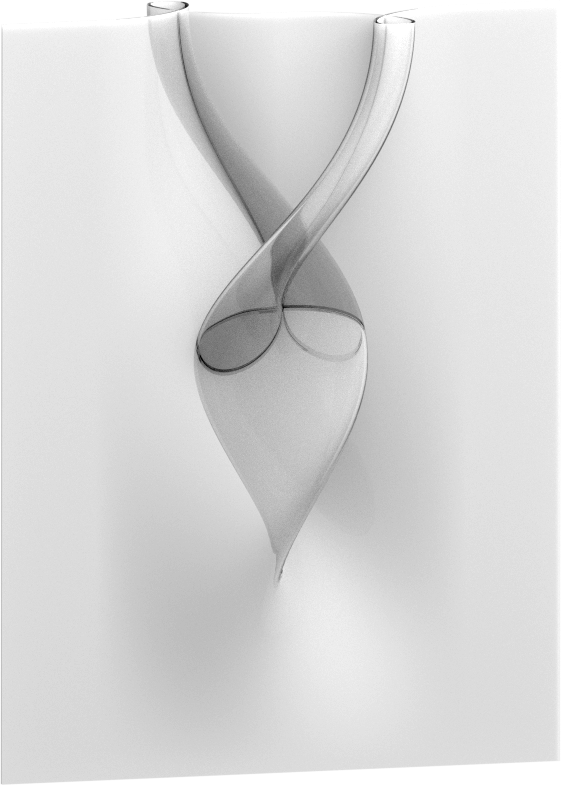}
\includegraphics[width=12cm]{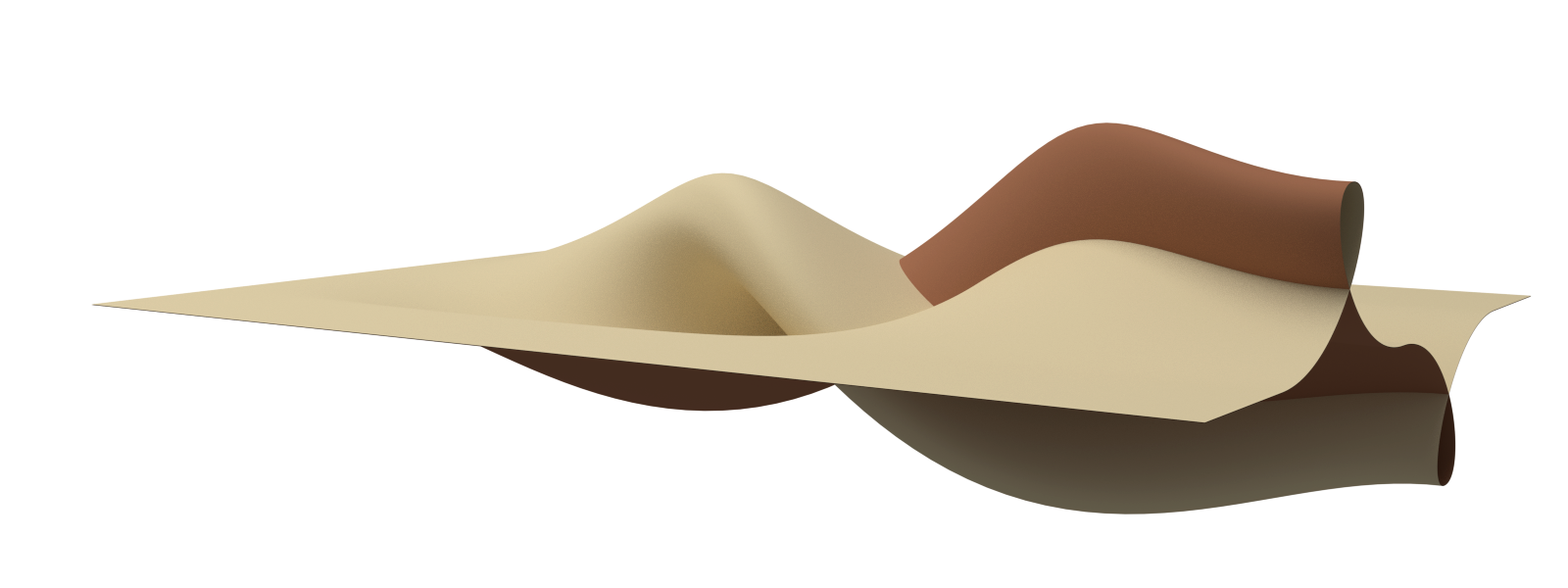}
\caption{The trace of a basic move. Above, left: The surface is represented white and opaque. Above, right: it is now semi-transparent so that one can see the hidden parts. The self-intersection curve is apparent. Below: the object is now presented horizontally and viewed from the side; this time the two faces of the surface have been given a different color.}
\label{fig:basicsurf}
\end{figure}

Somebody starting to try deforming immersions, creating more and more self intersection, will likely and rapidly end up with pieces diffeomorphic to this immersed surface patch.

Note that in our illustrations in Figures~\ref{fig:basicmove} and~\ref{fig:basicsurf} we chose to keep a central symmetry on the curves, which forces a tangency at the crossing of the curves at the moment of the triple point. 
However, the trace generated by the move on Figure~\ref{fig:basicmove} is robust as a $C^1$ embedding of the square in $\R^3$, thus any perturbation of the basic move has a diffeomorphic trace.
Hence there are several different\footnote{By different we mean: not diffeomorphic by a diffeomorphism preserving the $z$ coordinate.} ways to execute the move so that we still get a diffeomorphic trace, like on Figure~\ref{fig:basic-more}. Any such move will be called a \emph{basic move}. 
 
\clearpage

\begin{figure}[htbp]
\begin{tikzpicture}
\node at (0,0) {\includegraphics[scale=1.6,angle=0]{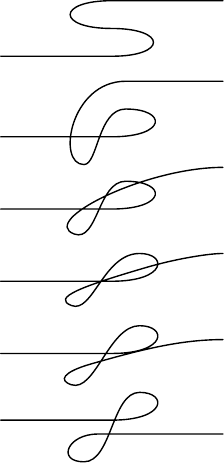}};
\node at (4,0) {\includegraphics[scale=1.6,angle=0]{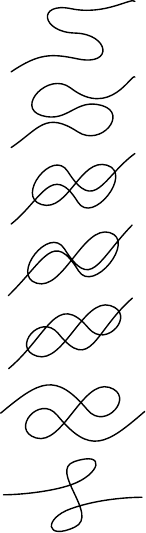}};
\end{tikzpicture}
\caption{Two other examples of realizations of a basic move. The trace is the same as in Figure~\ref{fig:basicmove} up to a diffeomorphism. In terms of a moving curve, the version on the left is more generic than the other two.}
\label{fig:basic-more}
\end{figure}

\section{Everting the sphere}

Let $S^2\in\R^3$ be a Euclidean sphere.

The eversion presented here is a three step process, starting from an embedding of the sphere $S^2$ in $\R^3$, ending with an embedding reversing orientation, and with two key positions $\cal S_1 =f_1(S^2)$ and $\cal S_2 =f_2(S^2)$ in between, where $f_1$ and $f_2$ are height preserving.

In steps 1 and 3, the deformation will stay in the space $\cal I_H$ height preserving immersions and will be provided by Proposition~\ref{prop:d}.
Step 2 will have to give up the height preserving property. On the bright side, it will only deform $f_1$ into $f_1$ by an ambient isotopy.
In particular the two key positions $\cal S_1$ and $\cal S_2$ are diffeomorphic.

\medskip 

\textbf{Note:} it is impossible to evert the sphere while only staying in height preserving immersions. Indeed, we remarked in Section~\ref{sec:tra} that, as one moves in $\cal I_H$, the orientation of each cap has to be preserved.

\medskip

To fix ideas and for reference, let us call: \emph{right} the $x>0$ direction, \emph{back} the $y>0$ direction, and \emph{up} the $z>0$ direction. Opposites are called the left, front and down directions.

Since the immersions in the two key positions are height preserving, each can be described as a capped trace (see Section~\ref{sec:tra}) of a family of $C^1$ immersions $\gamma_z$ of the circle in the plane such that $(s,z)\mapsto \gamma_z(s)$ is $C^1$ too. 

The family of curves associated to $\cal S_1$ is described as follows: a circle first undergoes a basic move as in the previous section on a small portion, let us say a small arc on the front side. Then the left loop just created is headed towards the inside of the circle and the right loop towards the outside.
In the next part of the motion we slide the outside loop along the circle until it comes back close to the inside loop, but now is on its left. We then do the basic move again, but backwards in time and with the left and right inverted. This ends with a circle.

Let us now describe the family associated to $\cal S_2$. It starts with a circle with orientation reversed, i.e.\ it is parameterized backwards by $S^1$.
The left and right side of initial circle first first pinch and cross: this gives a circle with two external loops, one on the front side and one on the back side. The back side loop will be left unchanged. On the front, we do a basic move next to the loop, on its right. After this movement, we see on the front a inside loop surrounded by two outside loops. We let the left and central loops do another basic move, backwards in time, with left and right inverted. We are left again with a circle with two outside loops, one on the front and one on the back. We uncross this exactly as the beginning of the family, but with time inverted.

See Figure~\ref{fig:S1S2}.
We leave it to the reader to check that, as well for $\cal S_1$ as for $\cal S_2$, the family can be taken such that $(s,z)\mapsto \gamma_z(s)$ is $C^1$. Continuity of $\partial \gamma/\partial s$ as a function of $(s,z)$ is the easiest to achieve. Concerning $\partial \gamma/\partial z$, let us suggest a hint (this is of course not the only way to proceed): the motion are defined by parts $z\in[z_k,z_{k+1}]$; a good way to ensure that $\partial \gamma/\partial z$ is continuous across horizontal seams $z=z_k$ is to slow down the motion, i.e.\ replacing $\gamma_z$ by $\gamma_{f(z)}$ for some $f$, so that  $\partial \gamma/\partial z$ vanishes entirely at $z=z_k$.  No slow-down should occur at the triple point, for it would prevent transversality (see the next paragraph).

We want the self intersections of each $\cal S_i$ to be transverse, including at the triple point: i.e.\ that the $2$ or $3$ tangent planes at a given self intersection point are in a transverse position. Then the immersions $f_i:S^2\to \cal S_i$ are \emph{robust}: every $C^1$-close perturbation is an immersion that is equivalent to $f_i$ by an ambient isotopy.\footnote{Recall that our definition of ambiant isotopy includes reparameterization of $S^2$ by an isotopy, c.f.\ just before Section~\ref{sec:imm}.}${}^,$\footnote{The reparameterization in the isotopy typically cannot be taken height preserving.}

\begin{figure}[htbp]
\includegraphics[scale=0.9]{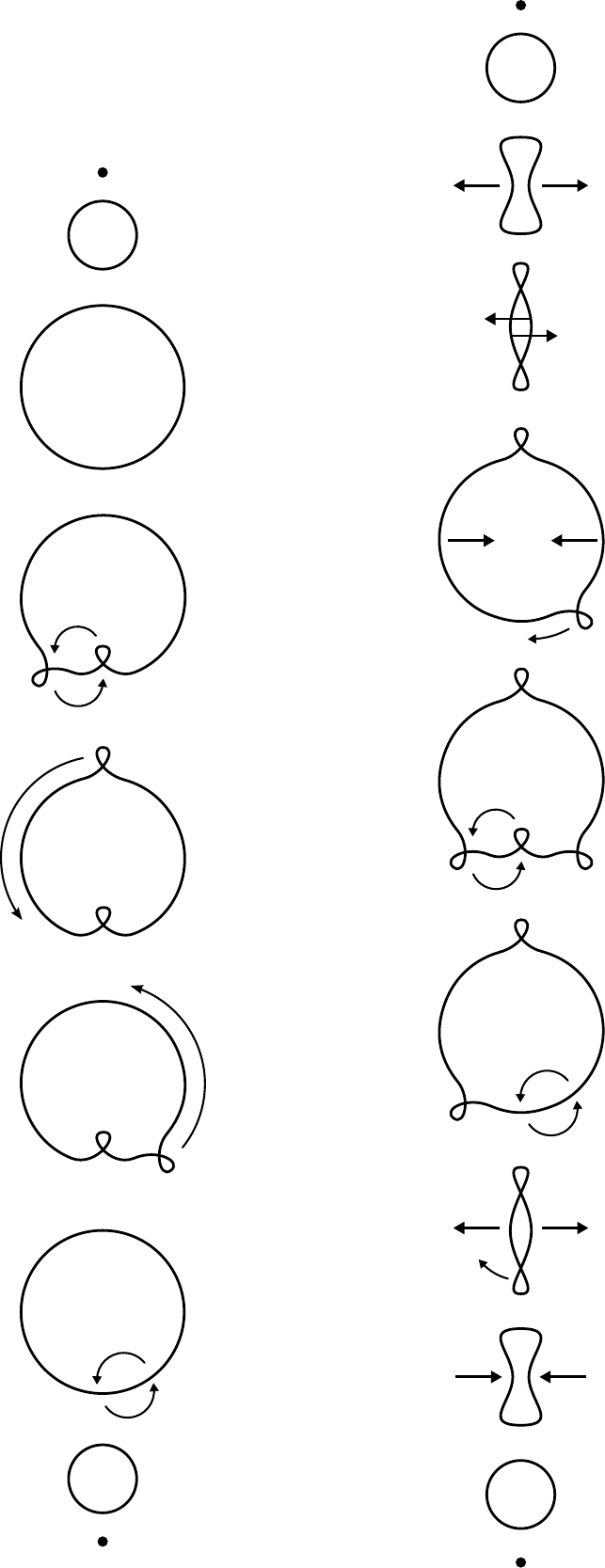}
\caption{The two moving curves whose traces define $\cal S_1$ (left) and $\cal S_2$ (right). Each sequence shall be read from bottom to top.}
\label{fig:S1S2}
\end{figure}

Let us call \emph{tubes} the traces of the loops: these are subsets of the surfaces. To be more precise, we include in the tube a whole connected component of the image of the immersion minus the self-intersection curve. Let us say that the tube is external if it includes an external loop.
Now observe that both $\cal S_1$ and $\cal S_2$ have only one external tube. Indeed, this is trivial for $\cal S_1$ and the tube containing the external loop on the back of $\cal S_2$ runs until the top, where it turns back down and joins the front external loop that is on the right of the internal loop. They then run in parallel up to the point, near the bottom, where they merge thanks to the basic move. Similarly ($\cal S_2$ is invariant by a rotation of $180\degree$ around the $y$-axis, which points front-to-back), the back external tube runs to the bottom, turns back up and joins the front external loop on the left of the internal loop; they run together and merge near the top with a basic move. 

It is now more or less clear that we can deform the external tube of $\cal S_1$ into the external tube of $\cal S_2$:  turning its part running on the back side by a little bit less than $180\degree$ (anticlockwise from the observers's viewpoint, clockwise if we prefer to look at it directly from the back side). While doing so, the part of the external tube that lies on the right has to be pulled up to the top of the sphere. Similarly, the left part is pushed down to the bottom.

Note that the bottom-most and top-most points of $\cal S_2$ lie on its external tube. Hence the caps are subsets of this tube. On this tube, the inside of the initial sphere is now outside. Hence the sides of $\cal S_1$ and $\cal S_2$ that point outwards at the caps are opposite sides of the immersed sphere if we match them by following the ambient isotopy from $\cal S_1$ to $\cal S_2$. See Figures~\ref{fig:sides} and~\ref{fig:S1S23D}. The height preserving unfoldings of $\cal S_1$ and $\cal S_2$ into embedded spheres, given by Proposition~\ref{prop:d}, both preserve the orientation of the caps.
This means that following the three Steps~1 then~2 then~3 everts the sphere.

\begin{figure}[htbp]
\includegraphics[scale=1.2]{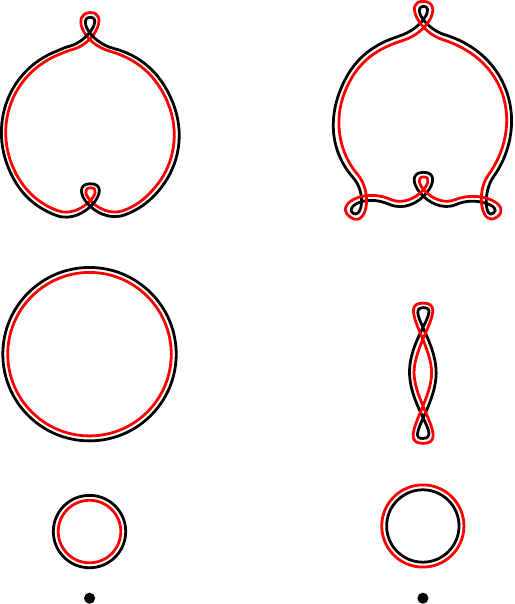}
\caption{Here are some of the curves that appeared in the bottom part of the previous figure. A parallel and very close red curve indicates the inside of the initial sphere; it can be thought of as a thin layer of paint. The parameterizations of each curve by $S^1$ is oriented so that the red paint is on the left of the oriented tangent vector. Look at the top row: during the ambient isotopy from $\cal S_1$ to $\cal S_2$, the part of the immersed sphere that is not close to the tubes retains its orientation, hence the red paint stays inside. Now when we descend the right column to turn this to a circle, we realize that the red paint ends up outside.}
\label{fig:sides}
\end{figure}

\begin{figure}[htbp]
\begin{tikzpicture}
\node at (0,0) {\includegraphics[width=10cm]{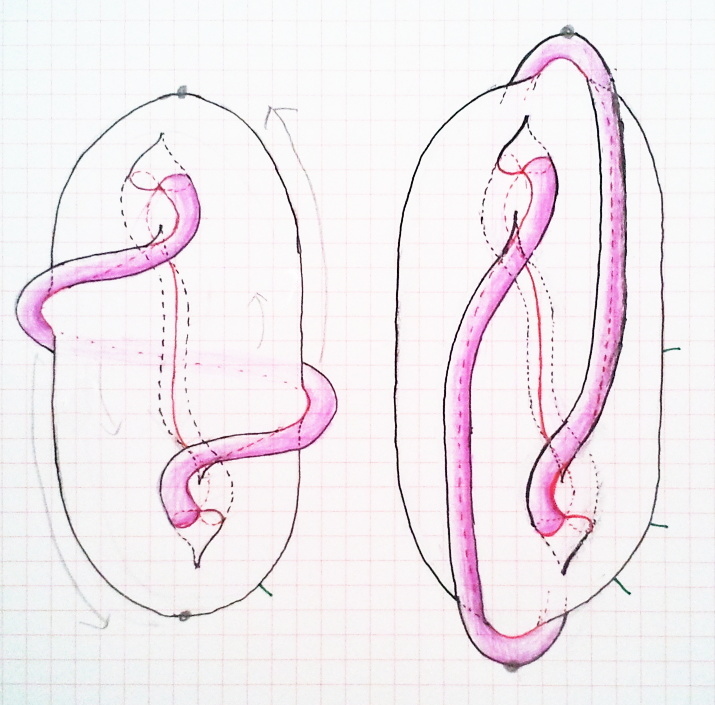}};
\node at (-4.2,-4.2) {$\cal S_1$};
\node at (4.2,-4.2) {$\cal S_2$};
\end{tikzpicture}
\caption{Hand drawn pictures of the key positions, i.e.\ the immersed surfaces $\cal S_1$ and $\cal S_2$. The face that correspond to the outside of the original sphere is white and the other face is purple. The red curves are the self-intersections. The black curves are the contour curves relative to the observer. Dashed lines represent hidden parts of these curves. For computer generated images of the same objects, see Figure~\ref{fig:S1S23D-comp}.}
\label{fig:S1S23D}
\end{figure}

\begin{figure}[htbp]
\begin{tikzpicture}
\node at (0,0) {\includegraphics[height=10cm]{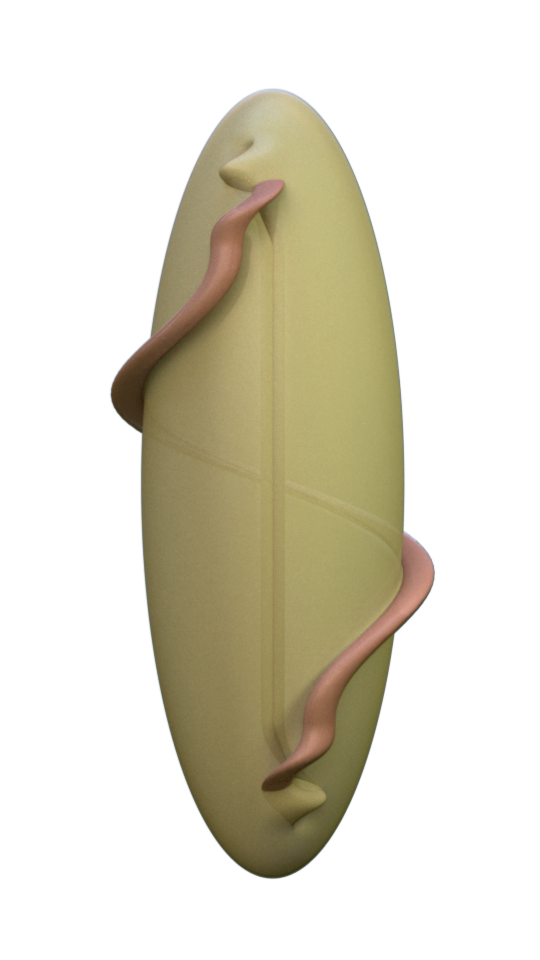}};
\node at (5.5,0) {\includegraphics[height=11cm]{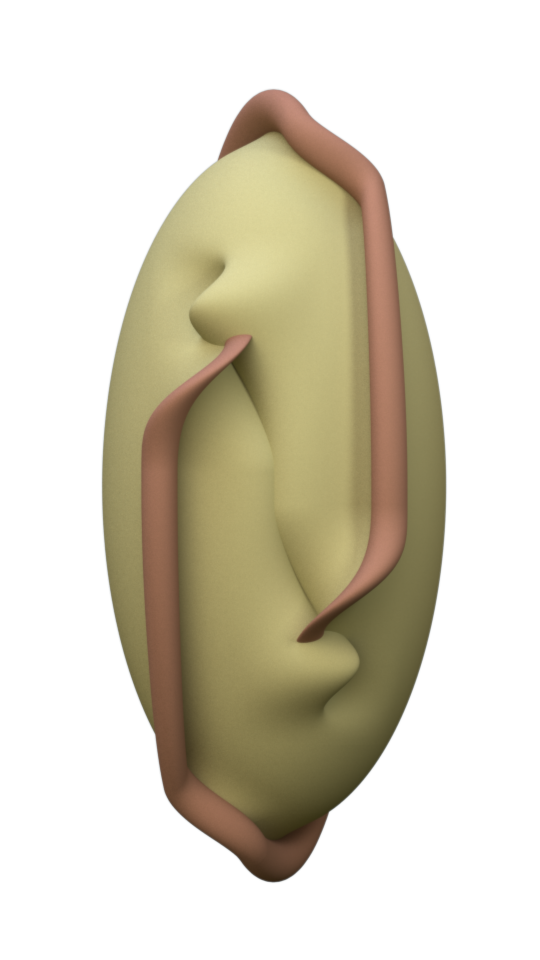}};
\node at (3,-8.5) {\includegraphics[height=7cm]{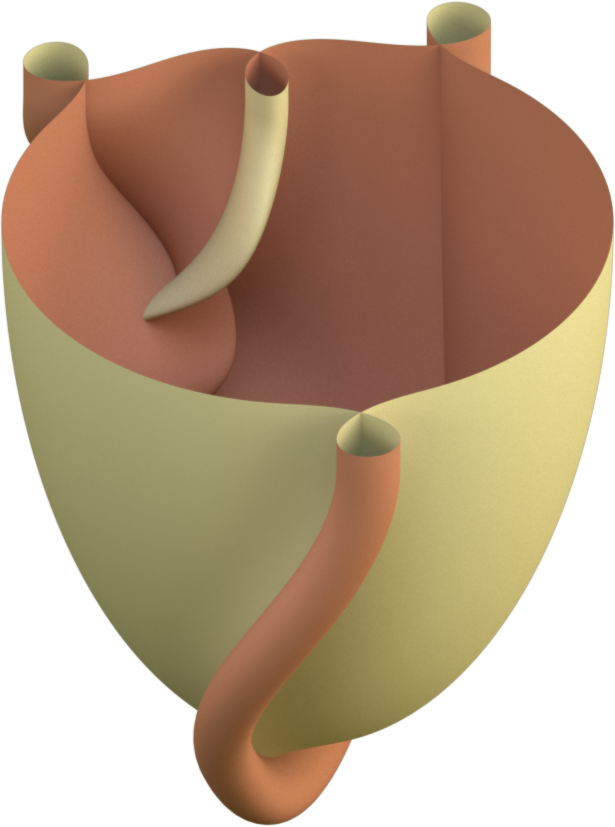}};
\node at (-2,-13.95) {\includegraphics[scale=0.3]{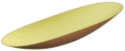}};
\node at (0.2,-13.78) {\includegraphics[scale=0.3]{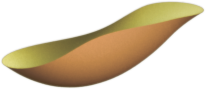}};
\node at (2.9,-13.68) {\includegraphics[scale=0.3]{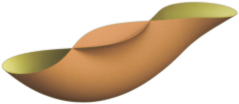}};
\node at (6.5,-13) {\includegraphics[scale=0.3]{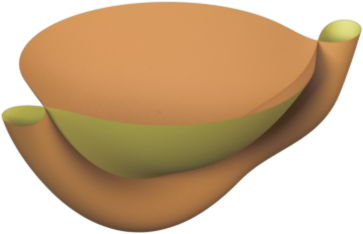}};
\end{tikzpicture}
\caption{Computer generated version of Figure~\ref{fig:S1S23D}. Left: we let object be slightly transparent so that the inner tube and the hidden part of the outer tube are visible. Right: the part of the external tube that is hidden goes nearly straight up along the back. Middle: $\cal S_2$ has been horizontally cut in half, so that we see the interior, and turned by a bit less than 180\degree\ so that the part of the external tube running on the back is more visible. Bottom: sections of $\cal S_2$ at different heights near the minimum.}
\label{fig:S1S23D-comp}
\end{figure}

\section{Postponed lemmas}

These technical lemmas are used in the proof that Lemma~\ref{lem:a} implies Proposition~\ref{prop:d}. They are proved in the appendix.

\subsection{Regularization} The following lemma is proved in Appendix~\ref{app:c1}.

\begin{lemma}\label{lem:c}
Assume that we have a path $t\in[0,1]\mapsto \gamma_t\in\cal I_n$ together with a homotopy $\wt h :[0,1]^2\to \cal I_n$ starting from this path: $\wt h(0,t) = \gamma_t$.
Assume moreover that:
\begin{itemize}
\item the map $(t,s)\mapsto \gamma_t(s)$ is $C^1$,
\item $\forall t$, $\wt h(1,t)$ is an embedding,
\item $\forall z$, $\wt h(z,0)$ is an embedding,
\item $\forall z$, $\wt h(z,1)$ is an embedding.
\end{itemize}
Then there exists another homotopy $h:(z,t)\mapsto h(z,t)\in\cal I_n$ with
\begin{itemize}
\item $\forall t$, $h(0,t)=\gamma_t$, (i.e.\ $h$ starts from the same path as $\wt h$)
\item $\forall z$, the map $h_z:(t,s)\mapsto h(z,t)(s)$ is $C^1$,
\item $z\mapsto h_z$ is continuous from $[0,1]$ to the Banach space $C^1([0,1]\times S^1,\R^2)$,
\item $\forall t$, $h(1,t)$ is an embedding,
\item $\forall z$, $h(z,0)$ is an embedding,
\item $\forall z$, $h(z,1)$ is an embedding.
\end{itemize}
\end{lemma}

We can even choose it so that $(z,t,s)\mapsto h(z,t)(s)$ is $C^1$, but this will not be used here.

\subsection{Recapping}

The following lemma is proved in Appendix~\ref{app:recap}.

\begin{lemma}\label{lem:recap}
Let $z_{\min}<z_0<z_1<z_{\max}$ be reals.
Let $S^2$ be a Euclidean sphere in $\R^3$ with extremal heights $z_{\min}$ and $z_{\max}$.
Let $\cal I_Z$ be the set of $C^1$ immersions from $S^2$ into $\R^3$ that preserve height, i.e.\ such that for all $z$, points at height $z$ are mapped to points at height $z$, and that are embeddings restricted to the set of points of $S^2$ with $z\leq z_0$ and similarly for $z\geq z_1$.
Let $\cal I_Z'$ be the set of $C^1$ immersions $S^2\cap (\R^2\times[z_0,z_1]) \to \R^3$ that preserve height and that are embeddings on the two boundary curves of the domain: $z=z_0$ and $z=z_1$. Note that domain restriction defines a map from $\cal I_Z$ to $\cal I_Z'$.
\begin{enumerate}
\item\label{item:1} (The restriction map is surjective.) Assume we are given a map $f \in \cal I_Z'$. Then there exists an extension $g\in \cal I_Z$.
\item\label{item:2} (Lifting property.) Assume we are given a map $g_0\in\cal I_Z$ and continuous path $f_t\in \cal I_Z'$ starting from the restriction of $g_0$.
Then there exists a family of extensions $g_t\in \cal I_Z$ of $f_t$, starting from $g_0$, such that $t\mapsto g_t$ is continuous.
\end{enumerate}
\end{lemma}

In the statement above, is to be understood that the sets $\cal I_Z$ and $\cal I_Z'$ are endowed with the topology of the corresponding Banach spaces of $C^1$ functions.

\part*{}
\renewcommand{\thesubsection}{\Alph{subsection}}
\stepcounter{section}
\section*{Appendix}

\subsection{Proof of Lemma~\ref{lem:a}}
Let us use \emph{s.d.}\ as an abbreviation for ``strong deformation''.

A proof of contractibility of $\cal I_n$ can be found in \cite{MiMu} (Theorem~2.10, easy modification) for $n\neq 0$ and in \cite{KoMi} (Proposition~2.1, easy modification) for $n=0$.
Note that in our application to sphere eversion, we use only the case $n=\pm1$. 
The construction is simple and explicit in the case $n\neq 0$, and in fact coincides with the Whitney construction \cite{Wh} in the case of constant speed curves.
As an open subset of a Banach space, contractibility of $\cal I_n$ implies s.d.-retractability to any of its point, so we are done.
However, it is also possible to directly write explicit s.d.-retractions to any given point, which we now do.

Let us parameterize the circle by $[0,1]$ with $1$ and $0$ identified: i.e.\ we see elements of $\cal I_n$ as maps from $[0,1]$ to $\C$.
The use of complex numbers below will only be a convenience.

To contract $\cal I_n$ to a single curve, the idea in short is to make a linear interpolation of the logarithm of the derivative. But the interpolated curves do not close up. This can be corrected by introducing a single corrugation (this method will not be explained here). Or if $n\neq 0$, we can use Whitney's trick (explained below). 

We now give the details in the case $n\neq 0$. Let $\alpha\in\cal I_n$ be the path to which we want to retract $\cal I_n$. By definition of $\cal I_n$, $\alpha'(0)$ is a positive real. Let $\cal I'_n$ be the set of maps $\gamma\in \cal I_n$ for which $\gamma'(0)=1$ and $\gamma(0)=0$.

We may assume for simplicity that $\alpha\in\cal I_n'$, the other cases being obtained from this one by composition with a similitude.

We first begin with a strong deformation retraction from $\cal I_n$ to $\cal I'_n$, i.e.\ a map $(\gamma,t)\in\cal I'_n\times [0,1] \to \gamma_t\in \cal I'_n$ that is continuous in the pair $(\gamma,t)$, with $\gamma_0=\gamma$, $\gamma_1\in \cal I'_n$ and with $\gamma\in \cal I'_n$ $\implies$ $\gamma_t=\gamma$. It is easy to find one: for instance use a rescaling as follows $\gamma_t = a_t.(\gamma -b_t)$ with $a_t=\exp(-t\log \gamma'(0))$ (recall that $\gamma'(0)$ is a positive real) and $b_t=t\gamma(0)$.

We are left to find a s.d.\ retraction of $\cal I'_n$ to $\alpha$, i.e.\ a map $(\gamma,t)\in\cal I'_n\times [0,1] \to \gamma_t\in \cal I'_n$ that is continuous in the pair $(\gamma,t)$, with $\gamma_0=\gamma$, $\gamma_1=\alpha$ and with $\gamma=\alpha$ $\implies$ $\gamma_t=\alpha$.
Let $g_\gamma:[0,1]\to\C$ be the unique determination of $\log \gamma'$ with $g_\gamma(0)=0$: $\exp g_\gamma(x)=\gamma'(x)$. Then $g_\gamma(1)=g_\alpha(1)=2\pi i n$.
 
The first attempt is to let $\gamma_t(x)$ be the curve such that $\gamma_t(0)=0$ and
\[ \gamma_t'(x)=\exp\big( (1-t) g_\gamma(x) + t g_\alpha(x)\big)
.\]
 This almost works: the map $(\gamma,t)\mapsto \gamma_t$ is well-defined, continuous, and we have $\gamma_0=\gamma$ and $\gamma_1=\alpha$. But for $t$ between $0$ and $1$, the curve $\gamma_t$ does not close up: in most cases $\gamma_t(0)\neq \gamma_t(1)$. (Note that still, $\gamma_t'(0)=\gamma_t'(1)$.)

The adapted Whitney's trick (see also \cite{MiMu}) is to define instead
\bEA 
f_t(x) & = & \exp\big( (1-t) g_\gamma(x) + t g_\alpha(x)\big)
\\
a_t &=& \int_0^1 f_t(x)dx \left/\int_0^1 |f_t(x)| dx\right..
\\
\gamma_t(0) &=& 0
\\
\gamma_t'(x) &=& f_t(x)- a_t |f_t(x)|
\eEA
This time the curve closes up. Note that if $\gamma=\alpha$ then $\forall t$, $f_t=\alpha'$ and $a_t=0$, thus $\gamma_t=\alpha$.
If $n\neq 0$ then $|a_t|<1$ and thus $\gamma_t$ is an immersion.
(If $n=0$ then the formula will fail at least in one case: $\gamma=\overline\alpha$ and $t=1/2$, since then $\forall x$, $f_t(x)>0$ thus $a_t=1$ thus $\forall x$, $\gamma'_t(x)= 0$.) As pointed out in \cite{MiMu} this formula has many advantages: it gives a smooth retraction and works right away in $C^\infty$.

\subsection{Proof of Lemma~\ref{lem:c}}\label{app:c1}
Notice that $\cal I_n$ is a subset of the Banach space $C^1(S^1,\R^2)$.
One hypothesis is that the initial path $\gamma_t(s)$ is $C^1$ in
the pair $(t,s)$.\footnote{This is weaker than than asking that the path $t\mapsto\gamma_t$ is a $C^1$ function of $t$ from $\R$ to the aforementioned Banach space. Indeed in the second case, we would require at least that $\frac{\partial}{\partial s}\frac\partial{\partial t} \gamma$ exists and be continuous in the pair $(t,s)$. But there are functions of $(t,s)$ that are $C^1$ in $(t,s)$ but do not satisfy this.}

Let us denote $\gamma^z_t=h(z,t)$. 
The continuity of the map $(z,t)\mapsto \gamma_t^z$ as a map from $[0,1]^2$ to the Banach space $C^1(S^1,\R^2)$ implies in particular that $\frac\partial{\partial s}\gamma_t^z(s)$ is a continuous function of $(z,t,s)$.
Consider a $C^1$ regularization $\rho(z,t,s)$ of the function $(z,t,s)\mapsto \gamma_t^u(s)$, obtained for instance by first taking a continuous extension to a neighborhood of $[0,1]^3$ in $\R^3$ whose $\partial/\partial s$ exists and is continuous\footnote{This is easy: let $f(z,t,s)=\gamma_t^z(s)$. First extend  $\partial f/\partial s$ by a $C^0$ map, and choose a $C^0$ map extending the slice $(z,t)\mapsto f(z,t,0)$. 
Then integrate.} and then convolving with a $C^1$ kernel.
This new map has no reason to be equal to $\gamma_t(s)$ for $z=0$. However, if the kernel is well-chosen, then the following quantities can be simultaneously taken arbitrarily small: \bEA
&&\sup_{z,t,s\in[0,1]} |\rho(z,t,s)-\gamma_t^z(s)|
\\
&&\sup_{z,t,s\in[0,1]} \left|\frac\partial{\partial s}\rho(z,t,s)-\frac\partial{\partial s}\gamma_t^z(s)\right|
\eEA
In particular (by the control on $\frac\partial{\partial s}$ and compactness of $[0,1]^3$) the curves $s\mapsto\rho(z,t,s)$ are still immersions. For the same reasons, we can ensure that for all $z$ and all $t$, the curves $s\mapsto\rho(z,0,s)$, $s\mapsto\rho(z,1,s)$ and $s\mapsto \rho(1,t,s)$ are still embeddings,
Now just complete the homotopy $\rho$ by concatenating it with a straight linear path $z\in[0,1]\mapsto\wt \rho_z$ in the Banach space $C^1([0,1]\times[0,1],\R^2)$ from $(t,s)\mapsto\gamma_t(s)$ to $(t,s)\mapsto \rho(0,t,s)$: i.e.\ let
\[\wt\rho(z,t,s)=\wt\rho_z(t,s)=(1-z)\gamma_t(s)+\rho(0,t,s).\]
By the control on $\partial \rho/\partial s$, if the corresponding quantities have been taken small enough, then for all $z,t$ the curves $s\mapsto\wt\rho(z,t,s)$ are still immersions and the top and bottom curves $s\mapsto\wt\rho(z,0,s)$ and $s\mapsto\wt\rho(z,1,s)$ are still embeddings.

Note: the concatenation will in most cases not be $C^1$ with respect to $z$ at the seam. It can be easily fixed (for instance by a change of variable on $z$ whose derivative vanishes at the seam), but this was not required.

\subsection{Embedded $C^1$ circles}\label{subsec:embed}

Here, by \emph{the circle}, or $S^1$, we mean $\R/\Z$. The following is well-known, cited here for reference:

\begin{proposition} Any compact and one dimensional $C^1$ differential manifold is diffeomorphic to the circle. 
\end{proposition}
As a consequence, any compact $C^1$ sub-manifold of $\R^2$ is the image of the circle $S^1$ by an embedding.
In this article, such an object is called a \emph{$C^1$ embedded circle}.

\begin{proposition}[$C^1$ Jordan-Shoenflies theorem]\label{prop:c1js}
For all $C^1$ embedded circle in $\R^2$ there is an orientation preserving $C^1$ diffeomorphism of $\R^2$ mapping it to a Euclidean circle.
\end{proposition}

See \cite{Mo}. On the set of $C^1$ diffeomorphism of $\R^2$, we put the topology of uniform convergence on compact subsets, of the maps and of their derivatives of order $1$. Inversion, i.e.\ the map $\phi\mapsto \phi^{-1}$, is continuous for this topology.

The following lemma is easy:
\begin{lemma}\label{lem:JF}
Any orientation preserving $C^1$ diffeomorphism of $\R^2$ is isotopic to the identity (by this we mean there is a continuous path to the identity in the set of $C^1$ diffeomorphisms). In particular it preserves the winding number of the tangent of parameterized curves. 
\end{lemma}
\noindent An example of proof is given by conjugating by rescalings: this deforms a $C^1$ diffeomorphism into its tangent map at the origin; now $GL^+(\R^2)$ is connected.

Using Proposition~\ref{prop:c1js} and Lemma~\ref{lem:JF} we get that for a $C^1$ embedding $S^1\to\R^2$, the tangent vector has winding number $\pm 1$.
Proposition~\ref{prop:c1js} is easily complemented as follows. Call \emph{canonical embedding} of $S^1$ in $\R^2$ the map $\gamma_{\mathrm{can}}:\R/\Z\to \R^2$, $x\mapsto(\cos(2\pi x),\sin(2\pi x))$.

\begin{proposition}[Complement]\label{prop:cpl}
For all $C^1$ embedding $\gamma: S^1\to\R^2$ with winding number $1$, the diffeomorphism in Proposition~\ref{prop:c1js}, call it $\phi$, can be chosen so that $\phi\circ\gamma$ is the canonical embedding.
\end{proposition}
\begin{proof}
Take any $\phi$ given by Proposition~\ref{prop:c1js} and so that the image is the unit circle. Then $f:=\phi\circ\gamma\circ\gamma_{\mathrm{can}}^{-1}$ is an orientation preserving $C^1$ diffeomorphism of the Euclidean circle, which can be extended to $\R^2$ into a $C^1$ diffeomorphism $F$ by an explicit formula, for instance as follows. The map $F$ will be the identity outside the annulus $B(0,3/2)-\ov B(0,1/2)$. Inside, we express $F$ it in polar coordinates:
\bEA
\text{write }\ f(1,\theta) &=& (1,\zeta(\theta))
\\
\text{let }\ F(r,\theta) &:=& (r,(1-\mu(r))\theta+\mu(r)\zeta(\theta))
\eEA
where $\mu\geq 0$ is a  $C^1$ function on $\R$ supported on $[1/2,3/2]$ and with $\mu(1)=1$.
Now replace $\phi$ by $F^{-1}\circ\phi$.
\end{proof}

These imply the corollary below, which is also a well-known result.
\begin{corollary}\label{cor:cnx}
For any pair of $C^1$ embeddings $\gamma_0,\gamma_1:S^1\to\R^2$ such that the winding number of the tangent vector is equal, there is an isotopy in the set of orientation preserving $C^1$ diffeomorphisms carrying one to the other as parameterized curves, i.e.: $\phi_t:\R^2\to\R^2$, $t\in[0,1]$, $\phi_0=\on{Id}_{\R^2}$, $\phi_1\circ \gamma_0=\gamma_1$.
\end{corollary}

For the next two statements, let us temporarily denote $\Emb=\Emb_{C^1}(S^1,\R^2)$ the set of $C^1$ embeddings of $S^1$ in $\R^2$. Recall that it is an open subset of the Banach space $C^1(S^1,\R^2)$ endowed with the norm $\max(\sup \|f\|,\sup\|f'\|)$.
Let $\gamma_0\in\Emb$.
Consider the map that, to a $C^1$ diffeomorphism $\phi$ of $\R^2$ associates its composition $\phi\circ\gamma_0 \in\Emb$. The two statement below are concerned with lifting properties of this map.

\begin{lemma}[Local lifting property]\label{lem:llp}
For all $\gamma_0\in\Emb$, there exists $r>0$ and a map $\gamma\mapsto\phi_\gamma$ from the ball $B(\gamma_0,r)$ in the Banach space $C^1(S^1,\R^2)$ to the set of $C^1$ diffeomorphisms of $\R^2$ such that:
\begin{itemize}
\item $B(\gamma_0,r)\subset \Emb$,
\item $\forall \gamma\in B(\gamma_0,r)$, $\phi_\gamma\circ \gamma_0=\gamma$,
\item $\gamma\mapsto \phi_\gamma$ is continuous,
\item $\phi_{\gamma_0}=\on{Id}_{\R^2}$.
\end{itemize}
\end{lemma}

\begin{proof}
Consider a tubular neighborhood $V$ of the image of $\gamma_0$, together with a trivialization, i.e.\ a $C^1$ map from $V$ to $C=\R/\Z\times\,]-\epsilon,\epsilon[$ sending the image of $\gamma_0$ to $\R/\Z \times \{0\}$. By this change of variable, it is enough to prove a similar statement on $C^1$ embeddings of $\R/\Z$ in $C$, with $\gamma_0$ being the trivial embedding, i.e.\ $\gamma_0(s)=(s,0)$, and getting $C^1$ diffeomorphisms of $C$ that are the identity outside $\R/\Z\times\,]-\epsilon/2,\epsilon/2[$.
Fix a family of functions $\zeta_a(y)$, with $a\in\,]-\epsilon/2,\epsilon/2[$ and $y\in\,]-\epsilon,\epsilon[$ such $\zeta_a(y)=y$ when $y\notin\,]-\epsilon/2,\epsilon/2[$, such that $\zeta_a(0)=a$, such that $(a,y)\mapsto \zeta_a(y)$ is $C^1$, such that $\zeta'_a>0$ for all $a$ and such that $\zeta_0=\on{Id}_C$. Existence is left to the reader, it can even be chosen $C^\infty$.
Now assume $\sup \|\gamma-\gamma_0\|\leq \epsilon/4$. Then the image of $\gamma_0$ is contained in $\R/\Z\times ]-\epsilon/4,\epsilon/4[$. 
Let $p_1,p_2:\R/\Z\times]-\epsilon,\epsilon[\to\R/\Z$ denote the first and second projections.
Assume $\sup\|\gamma'-\gamma'_0\|\leq 1/2$. Then $p_1\circ \gamma$ is a $C^1$ diffeomorphism of $\R/\Z$.
We let $\phi_\gamma(x,y)= (p_1\circ \gamma(x),\zeta_{p_2\circ \gamma(x)}(y))$.
\end{proof}

\begin{corollary}[Path lifting property]\label{cor:plp}
If $u\in[0,1]\mapsto\gamma_u$ is a continuous path in $\Emb$, then there exists a continuous path $u\mapsto\phi_u$ of $C^1$ diffeomorphisms  of $\R^2$, with $\phi_0=\on{Id}_{\R^2}$ and such that $\forall u$, $\phi_u\circ \gamma_0 = \gamma_u$.
\end{corollary}
\begin{proof}In short: by Lemma~\ref{lem:llp}, this can be done locally; 
we then take a finite cover of $[0,1]$ and compose the maps $\phi_\gamma$.
In details this gives: for all $u\in[0,1]$ there is a map $\gamma\mapsto\phi^u_\gamma$, continuous with respect to $\gamma$, with $\phi^u_{\gamma_u}=\on{Id}_{\R^2}$ and $\phi^u_{\gamma}\circ\gamma_u=\gamma$, defined for $\gamma$ close enough to $\gamma_u$, i.e.\ $d(\gamma,\gamma_u)<r_u$ for the distance on the Banach space $C^1$. Now there is some $\epsilon(u)>0$ such that $|u'-u|<\epsilon(u)$ $\implies$ $d(\gamma_u,\gamma_{u'})<r_u$ (uniform continuity over a compact interval).
There exists $\eta>0$ such that the following holds: $\forall u\in[0,1]$, $\exists v\in[0,1]$ with $B(u,\eta)\subset B(v,\epsilon(v))$. (The proof is a standard argument by contradiction and sequence extraction.) Choose $N>1/\eta$ and consider the finite sequence $u_k=k/N$ with $0\leq k \leq N$. We just saw that each $[u_k,u_{k+1}]$ is contained in some $B(v_k,\epsilon(v_k))$.
For $1\leq k\leq N$ let $\wh\phi_k=\phi^{u_{k-1}}_{\gamma_{u_{k}}}$; then $\wh\phi_k \circ \gamma_{u_{k-1}} = \gamma_{u_k}$.
Then for all $k$ with $0\leq k\leq N$ and for all $u\in[0,1]$ with $u_k\leq u\leq u_{k+1}$ we define
$\phi_u = \phi^{u_k}_{\gamma_u} \circ \wh\phi_{k} \circ \cdots \circ \wh\phi_{1}$ (this composition reduces to the leftmost term if $k=0$).
\end{proof}

\subsection{Proof of Lemma~\ref{lem:recap}}\label{app:recap}

Extension of $C^k$ functions of several variables, for any $k>0$, is trickier than it looks.
We recall the technique of symmetrization for extending a $C^1$ function on a half plane: assume $f$ is defined and continuous on the subset $x\leq 0$ of $\R^2$, that $f$ is $C^1$ on $x<0$ and that its partial derivatives have continuous extensions to $x\leq 0$. Then the extension of $f$ defined for  $x>0$ by $f(x,y)=2f(0,y)-f(-x,y)$ is $C^1$.
This technique can be used also to extend a map defined on a half-cylinder, or in a neighborhood of any differentiable curve where a local differentiable reflection can be defined.

We now explain how to add a bottom cap, the procedure is to be repeated to get the top cap. Recall that Lemma~\ref{lem:recap} has two statements.

\smallskip\noindent Proof of Statement~\eqref{item:1}: The restriction map is surjective.\par\smallskip

The bottom curve is parameterized by the Euclidean circle obtained by slicing the Euclidean sphere $S^2$ at height $z=z_0$.
By Proposition~\ref{prop:cpl}, there is a $C^1$ diffeomorphism $\phi$ of $\R^2$ mapping this bottom curve back to the identity on this circle. 
Extend it to $\R^3$ by setting $\Phi(x,y,z)=(\phi(x,y),z)$. It is enough to add a cap to $\Phi\circ f$ and to map it back by $\Phi^{-1}$.
Now if we just added the lower cap of $S^2$ using the identity, the surface we get would only be $C^0$ at the level of the seam: indeed the one-sided partial derivative along the meridians of the sphere just above the seam and just below the seam do not coincide. This can be solved with an explicit formula.
First consider a $C^1$ extension of the map $f$ to a neighborhood of the circle $z=z_0$. Call it $\wt f$ and say it is defined for $z>z_0-\epsilon$ for some $\epsilon>0$. Decrease $\epsilon$ so that $\wt f$ is an embedding on the part of $S^2$ with $z\in]z_0-\epsilon,z_0]$.  Consider a $C^1$ function $\chi :\R\to\R$ with $\chi(x)=0$ if $x\leq 0$ and $\chi(x)=1$ if $x\geq 1$.
Let us use the notation
\[ \on{lin}(a,b,t)=(1-t)a+tb
.\]
Now for $P\in S^2$ denote by $P_z$ its height coordinate, let $0<\eta<\epsilon/2$ be small and let
$g: S^2 \to \R^3$ be defined by
\bEA 
(P_z\geq z_0)\quad g(P) &=& f(P),
\\ (z_0> P_z> z_0-\eta)\quad g(P) &=& \on{lin}\left(\wt f(P),P,\chi\Big(\frac{z_0-z}{\eta}\Big)\right),
\\ (z_0-\eta\geq P_z)\quad g(P)&=& P.
\eEA
It preserves height and for $\eta$ small enough it is an embedding on the part $z\leq z_0$ of the sphere: this is trivially true below $z_0-\eta$ and between $z_0-\eta$ and $z_0$ we have a linear interpolation between the identity on a circle and a map $\wt f$ that is close to the identity for the $C^1$ norm if $\eta$ is small.

\smallskip\noindent Proof of Statement~\eqref{item:2}: Lifting property.\par\smallskip

Let $u\in[0,1]$ be the parameter for the continuous path $u\mapsto f_u$, with $f_u:(s,z) \to (\zeta_u(s,z),z)\in\R^2\times\R$, $s\in S^1$ and $z\in[z_0,z_1]$.

We first reduce to the case when the bottom curve does not change, as a parameterized curve, when $u$ changes: Corollary~\ref{cor:plp} applied to $\gamma_u(s)=\zeta_u(s,z_0)$ gives us an isotopy by $C^1$ diffeomorphisms $\phi_u$ of $\R^2$ with $\phi_0=\on{Id}_{\R^2}$ and $\phi_u\circ \gamma_0 = \gamma_u$. Let us extend this isotopy to $\R^3$ as $\Psi_u(x,y,z)=(\phi_u(x,y),z)\in\R^2\times \R$. We replace the family $f_u$ by $\Phi_u^{-1}\circ f_u$ and get a new family where the bottom curve is independent of $u$. We will then define a movement of its cap. Composing back by $\Phi_u$ will give the sought-for family $g_u$.

As in the proof of Statement~\eqref{item:1}, we use Proposition~\ref{prop:cpl} to get a $C^1$ diffeomorphism $\phi$ of $\R^2$ mapping the bottom curve to the identity on the circle $S^2\cap ``z=z_0"$ and extend it to $\R^3$ as $\Phi(x,y,z)=(\phi(x,y),z)$.

So we are reduced to the case where the bottom curve is independent of $u$ and is the identity on the circle $S^2\cap ``z=z_0"$. However, we still have a derivative along the meridians that depends on $u$, so following the cap is not trivial yet.

We can do it with a similar same perturbation method as in the proof of the other statement. Instead of interpolating linearly between an extension $\wt f_u$ of $f_u$ and the identity on $S^2$ we interpolate between $\wt f_u$ and the initial cap $g_0$. Then one has to see that a uniform (i.e.\ independent of $u$) value of $\eta$ can be chosen. Details are left to the reader.

\subsection{Proof of Proposition~\ref{prop:d}}

We are given a element $g\in\cal I_H$, i.e.\ a height preserving $C^1$ immersion from a Euclidean $S^2\subset\R^3$ to $\R^3$.

Let us remove caps from $g$: between the maximal and minimal heights $z_{\min}<z_{\max}$ of the image of $g$, there are some heights $z_0$ and $z_1$ with $z_{\min}<z_0<z_1<z_{\max}$ and such that the image of $g$ is an embedded surface below $z_0$ and above $z_1$. We now consider the part between $z_0$ and $z_1$.
To reparameterize this part by a cylinder, consider the height preserving cylindrical projection $p:S^1\times[z_0,z_1]\to S^2$ and let $f = g\circ p$, that we decompose as
\[ f(s,z) = (\gamma_z(s),z) \in \R^2\times \R
.\]

Recall that horizontal slices are $C^1$ immersed curves that depend continuously\footnote{As remarked earlier, the fact that $f$ is $C^1$ is stronger than this claim,
but it is weaker than asking that $z\mapsto\gamma_z$ is in a $C^1$ path in the Banach space $C^1(S^1,\R^2)$.} on the height: the map $z\mapsto \gamma_z$ is $C^0$ as a function from $[z_0,z_1]$ to the Banach space $C^1(S^1,\R^2)$, i.e.\ there is uniform convergence of $\gamma_z$ and of $\gamma_z'$ as $z$ converges to some height. The winding number $n$ of the tangent vector is thus independent of the height and since the curve is embedded for $z=z_0$, we get $n=\pm 1$.

The strategy is to apply Lemma~\ref{lem:a} to simultaneously untangle all these curves; then apply Lemma~\ref{lem:c} to smooth the result and finally Lemma~\ref{lem:recap} to recover the caps.

It shall be noted that to apply Lemma~\ref{lem:a} we would need the tangent vector at the basepoint to point in a constant direction, but this is not necessarily the case for the family $\gamma_z$. Also, to put back caps, Lemma~\ref{lem:recap} requires that the top and bottom curves remain embedded while we deform. The best way to do this is to ensure that the top and bottom curves are the fixed point of the retraction of Lemma~\ref{lem:a}. See below for the details.

For convenience, we now identify $\R^2$ with $\C$.
Consider the embedding
\[ \alpha:\left\{\begin{array}{l} \R/\Z\to\C\\ s\mapsto (e^{2\pi nis}-1)/2\pi in \end{array}\right
.\]
It is an element of $\cal I_n$. Better, we normalized it so that $\alpha(0)=0$ and $\alpha'(0)=1$. 

We choose an orientation preserving $C^1$ diffeomorphism $\phi_{z_0}$ of $\R^2$ such that $\phi_{z_0}\circ\gamma_{z_0} =\alpha$. Its existence follows from Proposition~\ref{prop:cpl} applied to $\gamma$ (if the winding number is $-1$, apply the proposition to $\gamma$ followed backwards).
Similarly we choose an orientation preserving $C^1$ diffeomorphism $\phi_{z_1}$ of $\R^2$ such that $\phi_{z_1}\circ\gamma_{z_1} =\alpha$.
By Corollary~\ref{cor:cnx} there is a continuous path $z\in[z_0,z_1]\mapsto\phi_z$  between $\phi_{z_0}$ and $\phi_{z_1}$ in the set of orientation preserving $C^1$ diffeomorphisms endowed with the topology of uniform convergence on compact sets of maps and of their first order differential. Let
\[
\Phi:\left\{\begin{array}{l} \R^2\times[z_0,z_1]\to \R^2\times[z_0,z_1]
\\
(x,y,z)\mapsto(\phi_z(x,y),z)\end{array}\right
.\]
Since we have not required $\phi_z$ to have a partial derivative w.r.t. $z$, the map $\Phi$ above is not necessarily globally $C^1$, even though it is $C^1$ on each horizontal plane.
Let
\newcommand{\va}{a}
\newcommand{\vb}{b}
\[ \va_z = \phi_z\circ\gamma_z
.\]
In particular $\va_{z_0}=\alpha$ and $\va_{z_1}=\alpha$.
Consider now the $\C$-affine maps
\[A_z: \zeta\in\C\mapsto \va_z'(0) \zeta + a_z(0).\]
Note that $A_{z_0}=\on{Id}_{\C}$ and $A_{z_1}=\on{Id}_{\C}$. Let now
\[ \vb_z = A_z^{-1}\circ \va_z
.\]
It shall be noted that for all $z$, $\vb_z\in\cal I_n$ and that $b_{z_0}=\alpha$ and $b_{z_1}=\alpha$.

Let now $t\in[0,1]\mapsto\Theta_t:\cal I_n\to\cal I_n$ be a strong deformation retraction to $\alpha$, as provided by Lemma~\ref{lem:a}.
We define for all $z\in[z_0,z_1]$:
\bEA
\vb^t_z &=& \Theta_t(\vb_z) , \\
\va^t_z &=& A_z\circ \vb^t_z , \\
\wt \gamma^t_z &=& \phi_z^{-1} \circ \va^t_z , \\
\wt f^t(s,z) &=& (\wt\gamma^t_z(s),z) .
\eEA
Then $\vb_{z_0}^t = \alpha$ and $\vb_{z_1}^t = \alpha$.
Also, for all $z$, $\vb^1_z=\alpha$. Hence for all $z,t$, the curves $\wt \gamma_{z_0}^t$, $\wt \gamma_{z_1}^t$ and
$\wt \gamma^1_z$ are embedded. Also, for all $z$, $\wt \gamma_z^0=\gamma_z$.
The hypotheses of Lemma~\ref{lem:c} are thus satisfied with $\wt h(z,t,s)=\wt f^t(s,z)$. There exists thus a homotopy $h(z,t,s)=h(z,t)(s)$ with
 $\forall t$, $h(0,t)=\gamma_t$, 
 $\forall z$, the map $h_z:(t,s)\mapsto h(z,t,s)$ is $C^1$,
 $z\mapsto h_z$ is continuous from $[0,1]$ to the Banach space $C^1([0,1]\times S^1,\R^2)$ and $\forall z,t$, the curves $h(1,t)$, $h(z,0)$ and $h(z,1)$ are embeddings.
Let now
\[
f^t(s,z) = (h(z,t,s),z).
\]

Let $S'\ =$ the Euclidean sphere $S^2$ we started from, minus its caps.
Then $f_t\circ p^{-1}$ is a family of height preserving $C^1$ immersion of $S'$ in $\R^3$. It varies continuously with $t$ in the corresponding Banach space. It starts from the restriction to $S'$ of the initial map $g\in \cal I_H$ and ends with an embedding of $S'$.

Now it is time to put back the caps. This is done using Lemma~\ref{lem:recap}, Statement~\eqref{item:2}. For the initial function $g_0$ in the hypotheses of the lemma we use $g$. For the functions $f_t$ of the lemma we use the functions $f_t\circ p^{-1}$ above.
The conclusion of the lemma gives us the homotopy $g_t\in \cal I_H$.

\bibliographystyle{alpha}
\bibliography{bib}

\end{document}